\documentclass[twoside,reqno,11pt]{amsart}

\usepackage{latexsym}

\newcommand{\qdn}{\hspace*{-1.5mm}}
\newcommand{\qqdn}{\hspace*{-2.5mm}}
\newcommand{\xqdn}{\hspace*{-5.0mm}}

\newcommand{\fns}{\footnotesize}

\newcommand{\ffnk}[4]{\left[\qdn\ba{#1}#3\\#4\ea{\!\Big|\:#2}\right]}

\newcommand{\binm}{\binom}

\newcommand{\nnm}{\nonumber}
\newcommand{\be}{\begin{equation}}
\newcommand{\ee}{\end{equation}}
\newcommand{\ba}{\begin{array}}
\newcommand{\ea}{\end{array}}
\newcommand{\bmn}{\begin{eqnarray}}
\newcommand{\emn}{\end{eqnarray}}
\newcommand{\bnm}{\begin{eqnarray*}}
\newcommand{\enm}{\end{eqnarray*}}
\newcommand{\bln}{\begin{subequations}}
\newcommand{\eln}{\end{subequations}}

\newtheorem{thm}{Theorem}

\newtheorem{corl}[thm]{Corollary}

\newtheorem{entry}{Entry}

\newcommand{\bbtm}[4]{\bibitem{kn:#1}{#2,}~{#3,}~{#4.}}
\newcommand{\cito}[1]{\cite{kn:#1}}
\newcommand{\citu}[2]{\cite[#2]{kn:#1}}

\begin{document} 
{\fns
\title{A new proof of Andrews' conjecture \\for $_4\phi_3$-series}

\author{$^a$Chuanan Wei, $^b$Xiaoxia Wang}
\dedicatory{
$^A$Department of Information Technology\\
  Hainan Medical College, Haikou 571199, China\\
  $^B$Department of Mathematics\\
  Shanghai University, Shanghai 200444, China}
\thanks{\emph{Email addresses}:
      weichuanan78@163.com (C. Wei), xwang 913@126.com (X. Wang)}

\address{ }


\begin{abstract}
In terms of Sear's transformation formula for $_4\phi_3$-series, we
give new proofs of a summation formula for ${_4\phi_3}$-series due
to Andrews \cito{andrews-b} and another summation formula for
${_4\phi_3}$-series conjectured in the same paper. Meanwhile, other
several related results are also derived.
 \end{abstract}

\maketitle\thispagestyle{empty}
\markboth{Chuanan Wei, Xiaoxia Wang}
         {A new proof of Andrews' conjecture for $_4\phi_3$-series}

\section{Introduction}

 For two complex numbers $x$ and
$q$, define the $q$-shifted factorial by
 \[(x;q)_0=1\quad\text{and}\quad (x;q)_n=\prod_{i=0}^{n-1}(1-xq^i)
  \quad\text{when}\quad n\in \mathbb{N}.\]
The fraction form of it reads as
\[\qqdn\qdn\ffnk{ccccc}{q}{a,&b,&\cdots,&c}{\alpha,&\beta,&\cdots,&\gamma}_n
=\frac{(a;q)_n(b;q)_n\cdots(c;q)_n}{(\alpha;q)_n(\beta;q)_n\cdots(\gamma;q)_n}.\]
 Following Gasper and Rahman \cito{gasper}, define the basic hypergeometric series by
\[_{1+r}\phi_s\ffnk{cccc}{q;z}{a_0,&a_1,&\cdots,&a_r}
{&b_1,&\cdots,&b_s}
 =\sum_{k=0}^\infty
\ffnk{ccccc}{q}{a_0,&a_1,&\cdots,&a_r}{q,&b_1,&\cdots,&b_s}_kz^k,\]
where $\{a_i\}_{i\geq0}$ and $\{b_j\}_{j\geq1}$ are complex
parameters such that no zero factors appear in the denominators of
the summand on the right hand side.  Then sear's transformation
formula for $_4\phi_3$-series (cf. \citu{gasper}{Equation (2.10.4)})
can be expressed as
 \bmn\qquad \label{sear-trans}
{_4\phi_3}\ffnk{ccccc}{q;q}{q^{-n},a,b,c}{d,e,q^{1-n}abc/de}
=\ffnk{ccccc}{q}{d/a,de/bc}{d,de/abc}_n
{_4\phi_3}\ffnk{ccccc}{q;q}{q^{-n},a,e/b,e/c}{e,de/bc,q^{1-n}a/d}.
 \emn

Recall Shapilo's identity (cf. \citu{koshy}{Equation (5.12)}):
\[\sum_{k=0}^nC_{2k}C_{2n-2k}=4^nC_n,\]
where $C_n=\frac{1}{n+1}\binm{2n}{n}$ are Catalan numbers. A deep
extension of the last formula due to Andrews \cito{andrews-b} can be
stated as
  \bmn \label{andrews-a}
{_4\phi_3}\ffnk{ccccc}{q;q}{q^{-n},a,b,q^{1/2-n}/ab}
 {q^{1-n}/a,q^{1-n}/b,q^{1/2}ab}
  =q^{-n/2}\ffnk{ccccc}{q}{ab}{a,b}_n\ffnk{ccccc}{q^{1/2}}{a,b,-q^{1/2}}{ab}_n.
 \emn
In the same paper, Andrews made the following conjecture:
 \bmn \label{andrews-b}
&&{_4\phi_3}\ffnk{ccccc}{q;q}{q^{-n},a,b,q^{3/2-n}/ab}
 {q^{1-n}/a,q^{2-n}/b,q^{1/2}ab} \nnm\\\nnm
  &&\:\:=\:\ffnk{ccccc}{q}{ab/q}{a,b/q}_n\ffnk{ccccc}{q^{1/2}}{a,b/q,-q^{1/2}}{ab/q}_n
 \frac{(q^{1/2}-ab)(q-bq^{n/2})}{q^{3/2}-ab^2q^n}\\
  &&\:\:\times\:\,\bigg\{\frac{q^{1-n/2}}{q-b}+\frac{ab(q^{1/2}-bq^{n/2})(q-abq^{n})}
   {(q^{1/2}-b)(q^{1/2}-abq^{n})(q-abq^{n/2})}\bigg\}.
 \emn

 Recently, Guo \cito{guo} confirmed \eqref{andrews-a} and
 \eqref{andrews-b} in accordance with the difference method.
A computer proof of \eqref{andrews-b} was also offered by Mu
\cito{mu} after reading a previous version of \cito{guo}. Inspired
by these work, we shall give not only new proofs of the two
 identities just mentioned but also several related results by using \eqref{sear-trans}.
\section{New proofs of Andrews' two summation formulas for $_4\phi_3$-series}
\begin{thm}\label{thm-a} For two complex numbers $\{a,b\}$ and a
nonnegative integer $n$, there holds
  \bnm
{_4\phi_3}\ffnk{ccccc}{q;q}{q^{-n},a,b,q^{1/2-n}/ab}
 {q^{1-n}/a,q^{-n}/b,q^{1/2}ab}
  =\ffnk{ccccc}{q}{ab}{a,qb}_n\ffnk{ccccc}{q^{1/2}}{a,q^{1/2}b,-q^{1/2}}{ab}_n.
 \enm
\end{thm}

\begin{proof}
By means of \eqref{sear}, we obtain the following relation:
 \bnm \label{sear}
{_4\phi_3}\ffnk{ccccc}{q;q}{q^{-n},a,b,q^{1/2-n}/ab}{q^{1-n}/a,q^{-n}/b,q^{1/2}ab}
&&\xqdn\!\!=\ffnk{ccccc}{q}{q^{1-n}/a^2,qab}{q^{1-n}/a,qb}_n\\
&&\xqdn\!\!\times\:{_4\phi_3}\ffnk{ccccc}{q;q}{q^{-n},a,q^{1/2}a,q^na^2b^2}{q^{1/2}ab,qab,a^2}.
 \enm
Evaluating the ${_4\phi_3}$-series on the right hand side by the
known identity (cf. \citu{andrews-a}{Equation (4.3)} and
\citu{gessel}{Equation (4.22)}):
 \bnm \label{known-iden}
{_4\phi_3}\ffnk{ccccc}{q;q}{q^{-n},a,q^{1/2}a,q^nc^2}{q^{1/2}c,qc,a^2}=
\ffnk{ccccc}{q^{1/2}}{-q^{1/2},q^{n/2}c,q^{-n/2}a/c}{-a,q^{1/2+n/2}c,q^{-n/2}/c}_n,
 \enm
we get Theorem \ref{thm-a} to complete the proof.
\end{proof}

\begin{thm}\label{thm-b} For three complex numbers $\{a,b,x\}$ and a
nonnegative integer $n$, there holds
  \bnm
\qquad{_5\phi_4}\ffnk{ccccc}{q;q}{q^{-n},qx,a,b,q^{1/2-n}/ab}
 {x,q^{1-n}/a,q^{1-n}/b,q^{1/2}ab}
  =\frac{q^{-n/2}-x}{1-x}\ffnk{ccccc}{q}{ab}{a,b}_n\ffnk{ccccc}{q^{1/2}}{a,b,-q^{1/2}}{ab}_n.
 \enm
\end{thm}

\begin{proof}
Performing the replacements $a\to b, b\to a$ in Theorem \ref{thm-a},
the result reads as
 \bnm
{_4\phi_3}\ffnk{ccccc}{q;q}{q^{-n},a,b,q^{1/2-n}/ab}
 {q^{-n}/a,q^{1-n}/b,q^{1/2}ab}
  =\ffnk{ccccc}{q}{ab}{qa,b}_n\ffnk{ccccc}{q^{1/2}}{q^{1/2}a,b,-q^{1/2}}{ab}_n.
 \enm
Combining the last formula with Theorem \ref{thm-a} by the following
relation:
 \bnm
 \frac{(qx;q)_k}{(x;q)_k}=\frac{1-xq^k}{1-x}
 &&\xqdn\!\!=\frac{(1-q^{-n}/a)(q^{-n}/b-x)}{(1-x)(q^{-n}/b-q^{-n}/a)}\frac{1-q^{k-n}/a}{1-q^{-n}/a}\\
 &&\xqdn\!\!+\:\frac{(1-q^{-n}/b)(q^{-n}/a-x)}{(1-x)(q^{-n}/a-q^{-n}/b)}\frac{1-q^{k-n}/b}{1-q^{-n}/b},
 \enm
we achieve Theorem \ref{thm-b} to finish the proof.
\end{proof}

When $x=0$, Theorem \ref{thm-b} reduces to Andrews' identity offered
by \eqref{andrews-a} exactly. Taking $x=q^{1/2-n}/ab$ in Theorem
\ref{thm-b}, we recover the known result due to Guo \citu{guo}{p.
1040}:
 \bmn \label{guo-a}
\:{_4\phi_3}\ffnk{ccccc}{q;q}{q^{-n},a,b,q^{3/2-n}/ab}
 {q^{1-n}/a,q^{1-n}/b,q^{1/2}ab}
  =\frac{1-abq^{n/2-1/2}}{1-abq^{n-1/2}}\ffnk{ccccc}{q}{ab}{a,b}_n\ffnk{ccccc}{q^{1/2}}{a,b,-q^{1/2}}{ab}_n.
 \emn

Other two related results are displayed as follows.

\begin{corl}[$x=q^{-1/2}ab$ in Theorem \ref{thm-b}]\label{corl-a}
  \bnm
{_4\phi_3}\ffnk{ccccc}{q;q}{q^{-n},a,b,q^{1/2-n}/ab}
 {q^{1-n}/a,q^{1-n}/b,q^{-1/2}ab}
  =q^{-n/2}\ffnk{ccccc}{q}{ab}{a,b}_n\ffnk{ccccc}{q^{1/2}}{a,b,-q^{1/2}}{q^{-1/2}ab}_n.
 \enm
\end{corl}

\begin{corl}[$x\to b/q$ and $b\to b/q$ in Theorem \ref{thm-b}]\label{corl-b}
  \bnm
\quad{_4\phi_3}\ffnk{ccccc}{q;q}{q^{-n},a,b,q^{3/2-n}/ab}
 {q^{1-n}/a,q^{2-n}/b,q^{-1/2}ab}
  =\frac{b-q^{1-n/2}}{b-q}\ffnk{ccccc}{q}{ab/q}{a,b/q}_n\ffnk{ccccc}{q^{1/2}}{a,b/q,-q^{1/2}}{ab/q}_n.
 \enm
\end{corl}

\begin{thm}\label{thm-c} For three complex numbers $\{a,b,x\}$ and a
nonnegative integer $n$, there holds
  \bnm
&&{_5\phi_4}\ffnk{ccccc}{q;q}{q^{-n},qx,a,b,q^{3/2-n}/ab}{x,q^{1-n}/a,q^{2-n}/b,q^{1/2}ab}\\
&&\:\:=\:\ffnk{ccccc}{q}{ab/q}{a,b/q}_n\ffnk{ccccc}{q^{1/2}}{a,b/q,-q^{1/2}}{ab/q}_n\frac{(q^{1/2}-ab)(q-bq^{n/2})}{(1-x)(q^{3/2}-ab^2q^n)}\\
&&\:\:\times\:\,\bigg\{\frac{q^{1-n/2}-xbq^{n/2}}{q-b}-\frac{(q^{1/2}-bq^{n/2})(q^{1/2}x-ab)(q-abq^{n})}
   {(q^{1/2}-b)(q^{1/2}-abq^{n})(q-abq^{n/2})}\bigg\}.
 \enm
\end{thm}

\begin{proof}
Combining \eqref{guo-a} with Corollary \ref{corl-b} by the following
relation:
 \bnm
 \frac{(qx;q)_k}{(x;q)_k}=\frac{1-xq^k}{1-x}
 &&\xqdn\!\!=\frac{(1-q^{1-n}/b)(x-q^{-1/2}ab)}{(1-x)(q^{1-n}/b-q^{-1/2}ab)}\frac{1-q^{k+1-n}/b}{1-q^{1-n}/b}\\
 &&\xqdn\!\!+\:\frac{(1-q^{-1/2}ab)(q^{1-n}/b-x)}{(1-x)(q^{1-n}/b-q^{-1/2}ab)}\frac{1-q^{k-1/2}ab}{1-q^{-1/2}ab},
 \enm
we attain Theorem \ref{thm-c} to complete the proof.
\end{proof}

When $x=0$, Theorem \ref{thm-c} reduces to Andrews' conjecture given
by \eqref{andrews-b} exactly. Letting $x\to a/q$ and $a\to a/q$ in
Theorem \ref{thm-c}, we recover the known result due to Guo
\citu{guo}{Equation (4.4)}:
 \bnm
\quad{_4\phi_3}\ffnk{ccccc}{q;q}{q^{-n},a,b,q^{5/2-n}/ab}
 {q^{2-n}/a,q^{2-n}/b,q^{-1/2}ab}
  &&\xqdn\!\!=\frac{q^{-n/2}(q^{3/2}-ab)}{q^{3/2}-q^nab}\\
  &&\xqdn\!\!\times\ffnk{ccccc}{q}{ab/q}{a/q,b/q}_n
  \ffnk{ccccc}{q^{1/2}}{q^{-1/2}a,q^{-1/2}b,-q^{1/2}}{q^{-3/2}ab}_n.
 \enm

Other two related results are laid out as follows.

\begin{corl}[$x=q^{-1/2}ab$ in Theorem \ref{thm-c}]\label{corl-c}
  \bnm
\quad{_4\phi_3}\ffnk{ccccc}{q;q}{q^{-n},a,b,q^{3/2-n}/ab}
 {q^{1-n}/a,q^{2-n}/b,q^{-1/2}ab}
  =q^{-n/2}\ffnk{ccccc}{q}{ab/q}{a,b/q}_n\ffnk{ccccc}{q^{1/2}}{a,q^{-1/2}b,-q^{1/2}}{ab/q}_n.
 \enm
\end{corl}

\begin{corl}[$x=q^{3/2-n}/ab$ in Theorem \ref{thm-c}]\label{corl-d}
  \bnm
&&{_4\phi_3}\ffnk{ccccc}{q;q}{q^{-n},a,b,q^{5/2-n}/ab}
 {q^{1-n}/a,q^{2-n}/b,q^{1/2}ab}\\
&&\:\:=\:\ffnk{ccccc}{q}{ab/q}{a,b/q}_n\ffnk{ccccc}{q^{1/2}}{a,b/q,-q^{1/2}}{ab/q}_n\frac{(q^{1/2}-ab)(q-bq^{n/2})}{(q^{3/2}-abq^n)(q^{3/2}-ab^2q^n)}\\
&&\:\:\times\:\,\bigg\{\frac{(q^{1/2}-bq^{n/2})(q+abq^{n/2})(q-abq^{n})}
   {(q^{1/2}-b)(q^{1/2}-abq^{n})}-\frac{q^{1+n/2}(q^{1/2}-a)}{1-q/b}\bigg\}.
 \enm
\end{corl}

According tho the linear methods of establishing Theorems
\ref{thm-b} and \ref{thm-c}, more related summation formulae can be
founded. We shall not offer the corresponding details.

\textbf{Acknowledgments}

 The work is supported by the Natural Science Foundations of China (Nos. 11301120, 11201241 and 11201291).


\end{document}